\documentclass[12pt]{amsart}

\usepackage{amsxtra,amssymb,amsmath,amscd,url,listings,amsrefs, stmaryrd}
\usepackage[utf8]{inputenc}
\usepackage{eucal}
\usepackage{fullpage}
\usepackage[colorlinks]{hyperref}
\usepackage{datetime}

\shortdate

\renewcommand{\leq}{\leqslant}
\renewcommand{\geq}{\geqslant}
 
\numberwithin{equation}{section}

\newcommand{\Cc}{\mathbf{C}}

\newcommand{\Rr}{\mathbf{R}}

\newcommand{\proba}{\mathbf{P}}
\newcommand{\expect}{\mathbf{E}}

\newcommand{\charfun}{\mathbf{1}}

\newcommand{\hol}{\mathcal{H}}

\newcommand{\ra}{\rightarrow}
\newcommand{\lra}{\longrightarrow}

\DeclareMathOperator{\Reel}{Re}

\DeclareMathOperator{\Tr}{Tr}

\newcommand{\eps}{\varepsilon}
\renewcommand{\rho}{\varrho}

\DeclareMathOperator{\SU}{SU}

\DeclareMathOperator{\SUhat}{\widehat{SU}_2}

\newcommand{\demi}{{\textstyle{\frac{1}{2}}}}

\DeclareMathSymbol{\gena}{\mathord}{letters}{"3C}
\DeclareMathSymbol{\genb}{\mathord}{letters}{"3E}

\def\sumb{\mathop{\sum \Bigl.^{\flat}}\limits}

\theoremstyle{plain}
\newtheorem{theorem}{Theorem}[section]
\newtheorem*{theorem*}{Theorem}
\newtheorem{lemma}[theorem]{Lemma}

\newtheorem{proposition}[theorem]{Proposition}

\theoremstyle{remark}

\theoremstyle{definition}

\newtheorem{remark}[theorem]{Remark}

\newcommand{\A}[1]{\mathsf{{#1}}}

\renewcommand{\geq}{\geqslant}
\renewcommand{\leq}{\leqslant}

\begin{document}

\title{Bagchi's Theorem for families of automorphic forms}

\author{E. Kowalski}
\address{ETH Zürich -- D-MATH\\
  Rämistrasse 101\\
  8092 Zürich, Switzerland} \email{kowalski@math.ethz.ch}
\thanks{Partially supported by a DFG-SNF lead agency program grant
  (grant 200021L\_153647).}


\begin{abstract}
  We prove a version of Bagchi's Theorem and of Voronin's Universality
  Theorem for family of primitive cusp forms of weight $2$ and prime
  level, and discuss under which conditions the argument will apply to
  general reasonable family of automorphic $L$-functions.
\end{abstract}

\maketitle

\section{Introduction}

The first ``universality theorem'' for Dirichlet series is Voronin's
Theorem~\cite{voronin} for the Riemann zeta function, which states
that for any $r<1/4$, any continuous function $\varphi$ defined and
non-vanishing on the disc $|s-3/4|\leq r$ in $\Cc$, which is
holomorphic in the interior, and any $\eps>0$, there exists $t\in\Rr$
such that
$$
\max_{|s-3/4|\leq r}{|\zeta(s+it)-\varphi(s)|}<\eps.
$$
In other words, up to arbitrary precision, any function $\varphi$ can
be approximated by some vertical translate of the Riemann zeta
function.
\par
Bagchi, in his thesis~\cite{bagchi}, provided a clear conceptual
explanation of this result, as the combination of two independent
statements:
\begin{itemize}
\item Viewing translates of the Riemann zeta function by $t\in [-T,T]$
  as random variables with values in a space of holomorphic function
  on the disc, Bagchi proves that these random variables converge in
  law, as $T\to +\infty$, to a natural random Dirichlet series, which
  is also expressed as a random Euler product;
\item Computing the support of the limiting random Dirichlet series,
  and checking that it contains the space of nowhere vanishing
  holomorphic functions on the disc, the universality theorem follows
  easily.
\end{itemize}

The key step, from our point of view, is the first part, which we call
\emph{Bagchi's Theorem}. Indeed, once the convergence in law is known,
it follows that there is ``some'' universality statement, with respect
to the functions in the support of the limiting random Dirichlet
series. The second step makes this support explicit. (This might be
compared with Deligne's Equidistribution Theorem, as applied to
families of exponential sums for instance: Deligne's Theorem shows
that there is always \emph{some} equidistribution of these sums.)
\par
The goal of this note is to give a first example of a genuinely
higher-degree statement of this type, and to deduce the corresponding
universality statement. We will also indicate a general principle that
should apply in many more cases.

\begin{theorem}[Universality in level aspect]\label{th-gl2}
  For $q$ prime $\geq 17$, let $S_2(q)^*$ be the non-empty\footnote{
    We assume $q\geq 17$ to ensure this property; it also holds for
    $q=11$.} finite set of primitive cusp forms for $\Gamma_0(q)$ with
  weight $2$ and trivial nebentypus. For $f\in S_2(q)^*$, let $L(f,s)$
  denote its Hecke $L$-function
$$
L(f,s)=\sum_{n\geq 1}\lambda_f(n)n^{-s},
$$
normalized so that the critical line is $\Reel(s)=\demi$.
\par
For any real number $r<\tfrac{1}{4}$, let $D$ be the open disc
centered at $3/4$ with radius $r$.  Then for any continuous function
$\varphi\,:\, \bar{D}\ra \Cc$ which is holomorphic and non-vanishing
in $D$ and satisfies
\begin{equation}\label{eq-condition}
\varphi(\sigma)>0\text{ for }\sigma\in D\cap \Rr,
\end{equation}
we have
$$
\liminf_{q\ra +\infty} \frac{1}{|S_2(q)^*|}|\{f\in S_2(q)^*\,\mid\,
\|L(f,\cdot)-\varphi\|_{\infty}<\eps \}|>0
$$
for any $\eps>0$, where the $L^{\infty}$ norm is the norm on
$\bar{D}$.
\end{theorem}

The main difference with previous results involving cusp forms (the
first one being due to Laurin\v cikas and Matsumoto~\cite{lm}) is that
we do not fix one such $L$-function $L(f,s)$ and consider shifts (or
twists) $L(f,s+it)$ or $L(f\times \chi,s)$, but rather we average over
the discrete family of primitive forms in $S_2(q)^*$. It is also
important to remark that the condition~(\ref{eq-condition}) is
necessary for a function on $D$ to be approximated by $L$-functions
$L(f,s)$ with $f\in S_2(q)^*$. (We will give more general statements
where the discs $D$ are replaced with more general compact sets in the
strip $\demi<\sigma<1$).

We will prove this Theorem in Section~\ref{sec-sketch}, after stating
the results generalizing the two steps of Bagchi's strategy for the
zeta function. The proof of Bagchi's Theorem for this family is an
analogue of a proof for the Riemann zeta function that is simpler than
Bagchi's proof (it avoids both the use of the ergodic theorem and any
tightness or weak-compactness argument).
\par
In Section~\ref{sec-general}, we discuss very briefly how this
strategy can in principle be applied to very general families of
$L$-functions, as defined in~\cite{families}.

\subsection*{Acknowledgements.} 

The simple proof of Bagchi's Theorem for the Riemann zeta funtion,
that we generalize here to modular forms, was found during a course on
probabilistic number theory at ETH Zürich in the Fall Semester 2015;
thanks are due to the students who attended this course for their
interest and remarks, and to B. Löffel for assisting with the
exercises (see~\cite{proba} for a draft of the lecture notes for this
course, especially Chapter~3).
\par
Thanks to the referee for carefully reading the text, and in
particular for pointing out a number of confusing mistakes in
references to the literature.

\subsection*{Notation.}

As usual, $|X|$ denotes the cardinality of a set.  By $f\ll g$ for
$x\in X$, or $f=O(g)$ for $x\in X$, where $X$ is an arbitrary set on
which $f$ is defined, we mean synonymously that there exists a
constant $C\geq 0$ such that $|f(x)|\leq Cg(x)$ for all $x\in X$. The
``implied constant'' is any admissible value of $C$. It may depend on
the set $X$ which is always specified or clear in context. We write
$f\asymp g$ if $f\ll g$ and $g\ll f$ are both true.
\par
We use standard probabilistic terminology: a probability space
$(\Omega,\Sigma,\proba)$ is a triple made of a set $\Omega$ with a
$\sigma$-algebra and a measure $\proba$ on $\Sigma$ with
$\proba(\Omega)=1$. We denote by $\expect(X)$ the expectation on
$\Omega$. The law of a random variable $X$ is the measure $\nu$ on the
target space of $X$ defined by $\nu(A)=\proba(X\in A)$. If
$A\subset \Omega$, then $\charfun_{A}$ is the characteristic function
of $A$.

\section{Equidistribution and universality for modular forms in the
  level aspect}
\label{sec-sketch}

We will prove Theorem~\ref{th-gl2} by combining the results of the
following two steps, each of which will be proved in a forthcoming
section.  Throughout, we assume that $q$ is a prime number $\geq 17$.
\par
\medskip
\par

\textbf{Step 1.} (Equidistribution; Bagchi's Theorem) For $q$ prime,
we view the finite set $S_2(q)^*$ as a probability space with the
probability measure proportional to the ``harmonic'' measure where
$f\in S_2(q)^*$ has weight
$$
\frac{1}{\langle f,f\rangle}
$$
in terms of the Petersson inner product. We write correspondingly
$\expect_q(\cdot)$ or $\proba_q(\cdot)$ for the corresponding
expectation and probability. Hence there exists a constant $c_q>0$
such that 
$$
\expect_q(\varphi(f))=\sum_{f\in S_2(q)^*}\frac{c_q}{\langle
  f,f\rangle} \varphi(f)
$$
for any $\varphi\colon S_2(q)^*\to\Cc$. From the Petersson formula, it
is known that $c_q\to 1/(4\pi)$ as $q\to +\infty$ (see, e.g., Iwaniec
and Kowalski~\cite[Ch. 14]{ik} or Cogdell and
Michel~\cite{cogdell-michel}).
\par
Let $D$ be a relatively compact open set in $\Cc$ such that $D$ is
invariant under complex conjugation.  Define $\hol(D)$ to be the
Banach space of functions holomorphic on $D$ and continuous and
bounded on $\bar{D}$, with the norm
$$
\|\varphi\|_{\infty}=\sup_{s\in\bar{D}}|\varphi(s)|.
$$
This is a separable complex Banach space. Define also $\hol_{\Rr}(D)$
to be the set of $\varphi\in \hol(D)$ such that
$f(\bar{s})=\overline{f(s)}$ for all $s\in D$ (this is well-defined
since $C$ is assumed to be invariant under conjugation). Note that the
$L$-function of $f$, restricted to $D$, is an element of
$\hol_{\Rr}(D)$ since the Hecke eigenvalues $\lambda_f(n)$ are real
for all $n\geq 1$.


We define $\A{L}_q$ to be the random variable $S_2(q)^*\to \hol(D)$
mapping $f\in S_2(q)^*$ to the restriction of $L(f,s)$ to $D$.  (This
depends on $D$, but the choice of $D$ will always be clear in the
context.)

If $\bar{D}$ is a compact subset of the strip $\demi<\Reel(s)<1$, then
we will show that $\A{L}_q$ converges in law to a random Dirichlet
series. To define the limit, let $(X_p)_p$ be a sequence of
independent random variables indexed by primes, taking values in the
matrix group $\SU_2(\Cc)$ and distributed according to the probability
Haar measure on $\SU_2(\Cc)$.

\begin{theorem}[Bagchi's Theorem for modular
  forms]\label{th-equidistribution}
  Assume that $\bar{D}$ is a compact subset of the strip
  $\demi<\Reel(s)<1$.  Then, as $q\ra +\infty$, the random variables
  $\A{L}_q$ converge in law to the random Euler product
$$
L_D(s)=\prod_{p}\det(1-X_pp^{-s})^{-1}= \prod_p
(1-\Tr(X_p)p^{-s}+p^{-2s})^{-1}
$$
which is almost surely convergent in $\hol(D)$, and belongs almost
surely to $\hol_{\Rr}(D)$.
\end{theorem}

\par
\medskip
\par
\textbf{Step 2.} (Support of the random Euler product) 

To deduce Theorem~\ref{th-gl2} from Theorem~\ref{th-equidistribution},
we need the following computation of the support of the limiting
measure.

\begin{theorem}\label{th-support}
  Suppose that $D$ is a disc with positive radius and diameter a
  segment of the real axis, always with $\bar{D}$ contained in
  $\demi <\Reel(s)<1$.  The support of the law of the random Euler
  product $L_D$ contains the set of functions $\varphi\in \hol(D)$
  such that $\varphi(x)>0$ for $x\in D\cap \Rr$.
\end{theorem}

Note that since $D\cap\Rr$ is an interval of positive length in $\Rr$,
the condition $\varphi(x)>0$ for all $x\in D\cap \Rr$ implies by
analytic continuation that $\varphi\in\hol_{\Rr}(D)$, which by
Bagchi's Theorem~\ref{th-equidistribution} is a necessary condition to
be in the support of $L_D$.

\par
\medskip
\par
\textbf{Step 3.} (Conclusion) The elementary Lemma~\ref{lm-elem}
below, combined with Theorems~\ref{th-equidistribution}
and~\ref{th-support}, implies Theorem~\ref{th-gl2} in the form
\begin{equation}\label{eq-harm-concl}
  \liminf_{q\to+\infty}
  \proba_q(\|L(f,\cdot)-\varphi\|_{\infty}<\eps)
  =
  \liminf_{q\to+\infty}
  \sum_{\substack{f\in S_2(q)^*\\
      \|L(f,\cdot)-\varphi\|_{\infty}<\eps }} \frac{c_q}{\langle
    f,f\rangle} >0
\end{equation}
for any function $\varphi$ as in Theorem~\ref{th-gl2} and any
$\eps>0$. We can easily deduce the ``natural density'' version from
this: let $A$ be the set of those $f\in S_2(q)^*$ such that
$\|L(f,\cdot)-\varphi\|_{\infty}<\eps$; then for any parameter
$\eta>0$, the definition of the harmonic measure on $S_2(q)^*$ gives
$$
\frac{1}{|S_2(q)^*|}\sum_{f\in A}{1}
=\expect_q\Bigl(\charfun_{A}(f)\frac{\langle
  f,f\rangle}{c_q|S_2(q)^*|}\Bigr)
\geq
\eta\Bigl(\proba_q(A)-\proba_q\Bigl(\frac{\langle
  f,f\rangle}{c_q|S_2(q)|^*}<\eta\Bigr)\Bigr).
$$
There exists $\delta>0$ such that the first term is $\geq \delta>0$
for all $q$ large enough by~(\ref{eq-harm-concl}); on the other hand,
a result of Cogdell and Michel~\cite[Cor. 1.16]{cogdell-michel} and the
classical relation between the Petersson norm and the symmetric square
$L$-function at $s=1$ (see, e.g.,~\cite[(5.101)]{ik}) imply that we
can find $\eta>0$ such that
$$
\lim_{q\to+\infty} \proba_q\Bigl(\frac{\langle
  f,f\rangle}{c_q|S_2(q)|^*}<\eta\Bigr) <\frac{\delta}{2}.
$$
For this value of $\eta$, we obtain
$$
\liminf_{q\to+\infty} \frac{1}{|S_2(q)^*|}\sum_{f\in A}{1}\geq
\frac{\eta\delta}{2}>0.
$$
More precisely, the result of Cogdell-Michel is that for any
$\eta>0$, we have
$$
\lim_{q\to+\infty} \proba_q(L(\mathrm{Sym}^2f,1)\leq \eta)=F(\log
\eta)
$$
where
$F$ is the limiting distribution function for the special value at
$1$ of the symmetric square $L$-function of $f\in
S_2(q)^*$. Since $F(x)\to 0$ when $x\to-\infty$, we obtain the result.

\begin{remark}
  It would also be possible to argue throughout with the uniform
  probability measure on $S_2(q)^*$; the only change would be a
  slightly different form of Theorem~\ref{th-equidistribution}, where
  the random variables $(X_p)$ would not be identically distributed
  (compare with the equidistribution theorems of Serre~\cite{serre}
  and Conrey--Duke--Farmer~\cite{cdf}).
\end{remark}

\begin{lemma}\label{lm-elem}
  Let $M$ be a separable complete metric space and $(X_n)$ a sequence
  of random variables with values in $M$ that converge in law to
  $X$. Let $S$ be the support of the law of $X$. Then for any
  $x\in S$, and any open neighborhood $U$ of $x$, we have
$$
\liminf_{n\ra +\infty}\proba(X_n\in U)>0.
$$
\end{lemma}

\begin{proof}
By classical criteria for convergence in law, we have
\begin{equation}\label{eq-1}
  \liminf_{n\ra +\infty}{\proba(X_n\in U)}\geq \proba(X\in U)
\end{equation}
for any open set $U\subset X$ (see, e.g.,~\cite[Th. 2.1
(iv)]{billingsley}). Since $x\in S$, we have $\proba(X\in U)>0$, hence
the result.
\end{proof}

\section{Proof of Theorem~\ref{th-equidistribution}}
\label{sec-proof1}

We begin with some preliminaries concerning the random Euler product
$L_D$. In fact, if will be convenient to view it as a holomorphic
function on larger domains then $D$, in a way that will be clear
below. For this purpose, we fix a real number $\sigma_0$ such that
$\demi<\sigma_0<1$, and such that the compact set $\bar{D}$ is
contained in the half-plane $S$ defined by $\Reel(s)>\sigma_0$.
\par
We recall that for $\nu\geq 0$, the $d$-th Chebychev polynomial is
defined by
$$
U_{\nu}(2\cos(x))=\frac{\sin((\nu+1)x)}{\sin(x)}.
$$
The importance of these polynomials for us lies in their relation with
the representation theory of $\SU_2(\Cc)$, namely
$$
U_{\nu}(2\cos(x))=\Tr\Bigl(\mathrm{Sym^{\nu}}\begin{pmatrix}
  e^{ix}&0\\0&e^{-ix}
\end{pmatrix}
\Bigr)
$$
for any $x\in\Rr$, where $\mathrm{Sym}^d$ is the $d$-th symmetric
power of the standard $2$-dimensional representation of $\SU_2(\Cc)$.
\par
We define a sequence of random variables $(Y_n)_{n\geq 1}$ by
$$
Y_n=\prod_{p^{\nu}\mid\mid n}U_{\nu}(\Tr(X_p)).
$$
In particular, we have $Y_nY_m=Y_{nm}$ if $n$ and $m$ are coprime, and
$Y_p=\Tr(X_p)$ if $p$ is prime. The sequence $(Y_p)$ is independent
and Sato-Tate distributed. Moreover, since $|U_{\nu}(t)|\leq \nu+1$
for all $\nu\geq 0$ and all $t\in\Rr$, we have
$$
|Y_n|\leq \prod_{p^{\nu}\mid\mid n}(\nu+1)=d(n)\ll n^{\eps}
$$
for $n\geq 1$ and $\eps>0$, where the implied constant depends only on
$\eps$.

\begin{lemma}\label{lm-random}
  \emph{(1)} Almost surely, the random Euler product
$$
\prod_{p}\det(1-X_pp^{-s})^{-1}
$$
converges and defines a holomorphic function on $S$. In particular,
$L_D$ converges almost surely to define an $\hol(D)$-valued random
variable.
\par
\emph{(2)} Almost surely, we have
$$
\prod_{p}\det(1-X_pp^{-s})^{-1}=\sum_{n\geq 1}Y_nn^{-s}
$$
for all $s\in S$, and in particular $L_D$ coincides with the random
Dirichlet series on the right.
\par
\emph{(3)} For $\sigma>1/2$ and $u\geq 2$, we have
$$
\expect\Bigl(\Bigl|\sum_{n\leq u}Y_nn^{-\sigma}\Bigr|^2\Bigr)
\ll 1,
$$
where the implied constant depends only on $\sigma$.
\end{lemma}

\begin{proof}
  (1) Let $\sigma$ be a fixed real number such that
  $\demi<\sigma<\sigma_0$. By expanding, we can write
$$
-\log\det(1-X_pp^{-s})=Y_pp^{-s}+g_p(s)
$$
where  the random series
$$
\sum_p g_p(s)
$$
converges absolutely (and surely) for $\Reel(s)>1/2$. Since
$\expect(Y_pp^{-\sigma})=0$ and
$\expect(Y_p^2p^{-2\sigma})=p^{-2\sigma}$, Kolmogorov's Three Series
Theorem (see, e.g.,~\cite[Th. 0.III.2]{li-queffelec}) implies that the
random series
$$
\sum_{p} Y_pp^{-\sigma}
$$
converges almost surely. By well-known results on Dirichlet series,
this means that the random series
$$
\sum_p Y_pp^{-s}
$$
converges almost surely to a holomorphic function on the half-plane
$S$. This implies the first statement by taking the exponential. The
second follows by restricting to $D$ since $\bar{D}$ is contained in
$S$.
\par
(2) We first show that the almost surely the random Dirichlet series
$$
\tilde{L}(s)=\sum_{n\geq 1}Y_nn^{-s}
$$
converges and defines a function holomorphic on $S$.  The key point is
that the variables $Y_n$ for $n$ squarefree form an orthonormal
system: we have
$$
\expect(Y_nY_m)=\delta(n,m)
$$
if $n$ and $m$ are squarefree numbers. Indeed, if $n\not=m$, there is
a prime $p$ dividing only one of $n$ and $m$, say $p\mid n$, and then
by independence we get
$\expect(Y_nY_m)=\expect(Y_p)\expect(Y_{n/p}Y_m)=0$; and if $n=m$ is
squarefree then we have
$$
\expect(Y_n^2)=\prod_{p\mid n}\expect(Y_p^2)=1.
$$
Fix again $\sigma$ such that $\demi<\sigma<\sigma_0$. By the
Rademacher--Menchov Theorem (see, e.g.~\cite[Th. B.8.4]{proba}), the
random series
$$
\sumb_{n}Y_nn^{-\sigma}
$$
over squarefree numbers converges almost surely. By elementary
factorization and properties of products of Dirichlet series (the
product of an absolutely convergent Dirichlet series and a convergent
one is convergent, see e.g.~\cite[Th. 54]{hardy-riesz}) the same holds
for
$$
\sum_{n\geq 1} Y_nn^{-\sigma}.
$$
As in (1), this gives the almost sure convergence of the series
defining $\tilde{L}(s)$ to a holomorphic function in $S$. Restricting
gives the $\hol(D)$-valued random variable $\tilde{L}_D$.
\par
Finally, almost surely both the random Euler product and the random
Dirichlet series converge and are holomorphic for
$\Reel(s)>\sigma_0$. For $\Reel(s)>3/2$, they converge absolutely, and
coincide by a well-known formal Euler product computation: for any
prime $p$ and any $x\in\Rr$, denoting
$$
t(x)=\begin{pmatrix}e^{ix}&0\\0&e^{-ix}
\end{pmatrix}
$$
we have
\begin{align*}
  \det(1-t(x)p^{-s})^{-1}
  &=
    (1-e^{ix}p^{-s})^{-1}(1-e^{-ix}p^{-s})^{-1}\\
  &=
    \sum_{\nu\geq 0}\Tr\mathrm{Sym}^{\nu}(t(x))p^{-\nu s}= \sum_{\nu\geq
    0}U_{\nu}(x)p^{-\nu s}
\end{align*}
(compare the discussion of Cogdell and Michel in~\cite[\S
2]{cogdell-michel}).  By analytic continuation, we deduce that
$L_D=\tilde{L}_D$ almost surely as $\hol(D)$-valued random variables.
\par
(3) Since the random varibles $Y_n$ are real-valued, we have
$$
\expect\Bigl(\Bigl|\sum_{n\leq u}Y_nn^{-\sigma}\Bigr|^2\Bigr)
=\sum_{n,m\leq u}\frac{1}{(nm)^{\sigma}}\expect(Y_nY_m).
$$
For given $n$ and $m$, let $d=(n,m)$ and $n'=n/d$, $m'=m/d$. Then by
multiplicativity and independence of the variables $(Y_p)_p$, we have
$$
\expect(Y_nY_m)=\expect(Y_d^2)\expect(Y_{n'})\expect(Y_{m'}).
$$
By the definition of $Y_p$, we have $\expect(Y_{n'})=0$ if $n'$ is
divisible by a prime $p$ with odd exponent, and similarly for
$\expect(Y_{m'})$. Hence we have $\expect(Y_nY_m)=0$ unless both $n'$
and $m'$ are squares. Therefore
$$
\expect\Bigl(\Bigl|\sum_{n\leq u}Y_nn^{-\sigma}\Bigr|^2\Bigr) \leq
\sum_{d\geq 1}\frac{\expect(Y_d^2)}{d^{2\sigma}} \sum_{m,n\geq
  1}\frac{1}{(mn)^{2\sigma}}\expect(Y_mY_n)<+\infty
$$
since $\sigma>1/2$ and $\expect(Y_n)\ll n^{\eps}$ for any $\eps>0$.
\end{proof}

The key arithmetic properties of the family $S_2(q)^*$ of modular
forms that are required in the proof of
Theorem~\ref{th-equidistribution} are the following:

\begin{proposition}[Local spectral equidistribution]\label{pr-spectral}
  As $q\to+\infty$, the sequence $(\lambda_f(p))_{p}$ of Fourier
  coefficients of $f\in S_2(q)^*$ converges in law to the sequence
  $(Y_p)_p$.
\end{proposition}

\begin{proof}
  This is a well-known consequence of the Petersson formula, see
  e.g.~\cite[Prop. 8]{klsw},~\cite[Appendix]{kst2}
  or~\cite[Prop. 1.9]{cogdell-michel}; here restricting to prime level
  $q$ and weight $2$ also simplifies matters since this ensures that
  the old space of $S_2(q)$ is zero.
\end{proof}

\begin{proposition}[First moment estimate]\label{pr-second-moment}
There exists an absolute constant $A\geq 1$ such that for any real
number $\delta>0$ with $\delta<1/2$, and for any $s\in\Cc$ such that
$\demi+\delta\leq\Reel(s)$, we have
$$
\expect_q(|L(f,s)|)\ll (1+|s|)^A,
$$
where the implied constant depends only on $\delta$.
\end{proposition}

\begin{proof}
  This follows easily, using the Cauchy-Schwarz inequality, from the
  second moment estimate~\cite[Prop. 5]{km} of Kowalski and Michel
  (with $\Delta=0$); although this statement is not formally the same,
  it is in fact a more difficult average (it operates closer to the
  critical line).
\end{proof}

We now prove some additional lemmas.

\begin{lemma}[Polynomial growth]\label{lm-pol-growth}
  For any real number $\sigma>\sigma_0$, we have
$$
\expect\Bigl( \Bigl|\sum_{n\geq 1}Y_nn^{-s}\Bigr|\Bigr)\ll 1+|s|
$$
uniformly for all $s$ such that $\Reel(s)\geq \sigma>\sigma_0$.
\end{lemma}

\begin{proof}
We write
$$
L(s)=\sum_{n\geq 1}Y_nn^{-s}.
$$
This is almost surely a function holomorphic on the half-plane $S$.
The series
$$
\sum_{n\geq 1}\frac{Y_n}{n^{\sigma_0}}
$$
converges almost surely. Therefore the partial sums
$$
S_u=\sum_{n\leq u}\frac{Y_n}{n^{\sigma_0}}
$$
are bounded almost surely. By summation by parts, it follows from the
convergence of the series $L(s)$ that for any $s$ with real part
$\Reel(s)\geq \sigma>\sigma_0$, we have
$$
L(s)=(s-\sigma_0)\int_1^{+\infty} \frac{S_u}{u^{s-\sigma_0+1}}du,
$$
where the integral converges almost surely. Hence almost surely
$$
|L(s)|\leq (1+|s|)\int_1^{+\infty}
\frac{|S_u|}{u^{\sigma-\sigma_1+1}}du.
$$
Fubini's Theorem and the Cauchy-Schwarz inequality then imply
\begin{align*}
  \expect(|L(s)|)&\leq (1+|s|) \int_1^{+\infty}
                   \expect(|S_u|)\frac{du}{u^{\sigma-\sigma_0+1}} \\
                 &\leq (1+|s|)\int_1^{+\infty}\expect(|S_u|^2)^{1/2}
                   \frac{du}{u^{\sigma-\sigma_0+1}}\ll 1+|s|
\end{align*}
by Lemma~\ref{lm-random} (3).
\end{proof}

We now consider some elementary approximation statements of the
$L$-functions and of the random Dirichlet series by smoothed partial
sums. For this, we fix once and for all a smooth function
$\varphi\colon [0,+\infty[\to [0,1]$ with compact support such that
$\varphi(0)=1$, and we denote $\hat{\varphi}$ its Mellin transform.
\par
We also fix $T\geq 1$ and a compact interval $I$ in $]1/2,1[$ such
that the compact rectangle $R=I\times [-T,T]\subset \Cc$ is contained
in $S$ and contains $D$ in its interior. We then finally define
$\delta>0$ so that
$$
\min\{\Reel(s)\,\mid\, s\in R\}= \frac{1}{2}+2\delta.
$$

\begin{lemma}\label{lm-random-critical}
  For $N\geq 1$, define the $\hol(D)$-valued random variable
$$
L_{D}^{(N)}=\sum_{n\geq 1}Y_n\varphi\Bigl(\frac{n}{N}\Bigr)n^{-s}.
$$
\par
We then have
$$
\expect(\|L_{D}-L_{D}^{(N)}\|_{\infty})\ll N^{-\delta}
$$
for $N\geq 1$, where the implied constant depends on $D$.
\end{lemma}

\begin{proof}
We again write
$$
L(s)=\sum_{n\geq 1}Y_nn^{-s}
$$
when we wish to view the Dirichlet series as defined and holomorphic
(almost surely) on $S$. 
\par
For any $s$ in the rectangle $R$, we have almost surely the
representation
\begin{equation}\label{eq-mellin-2}
  L(s)-L^{(N)}(s)=-\frac{1}{2i\pi}\int_{(-\delta)}
  L(s+w)\hat{\varphi}(w)N^wdw
\end{equation}
by standard contour integration.\footnote{\ Here and below, it is
  important that the ``almost surely'' property holds for \emph{all}
  $s$, which is the case because we work with random holomorphic
  functions, and not with particular evaluations of these random
  functions at specific points $s$.}
\par
We also have almost surely for any $v$ in $D$ the Cauchy formula
$$
L_D(v)-L_{D}^{(N)}(v)=\frac{1}{2i\pi}\int_{\partial R}
(L(s)-L^{(N)}(s))\frac{ds}{s-v},
$$
where the boundary of $R$ is oriented counterclockwise. The definition
of the rectangle $R$ ensures that $|s-v|^{-1}\gg 1$ for $v\in D$ and
$s\in \partial R$, and therefore
$$
\|L_D-L_{D}^{(N)} \|_{\infty} \ll \int_{\partial R}
|L(s)-L^{(N)}(s)|\, |ds|.
$$
Using~(\ref{eq-mellin-2}) and writing $w=-\delta+iu$ with $u\in\Rr$,
we obtain
$$
\|L_D-L_{D}^{(N)} \|_{\infty}\ll N^{-\delta}\int_{\partial
  R}\int_{\Rr} |L(-\delta+iu+s)| \ |\hat{\varphi}(-\delta+iu)||ds|du.
$$
Therefore, taking the expectation, and using Fubini's Theorem, we get
\begin{align*}
  \expect(\|L_D-L_{D}^{(N)}\|_{\infty}) 
  &\ll N^{-\delta} \int_{\partial
    R}\int_{\Rr} \expect\bigl(|L(-\delta+iu+s)|\bigr) \
    |\hat{\varphi}(-\delta+iu)||ds|du\\
  & \ll N^{-\delta} \sup_{s=\sigma+it\in
    R} \int_{\Rr} \expect\bigl(|L(-\delta+iu+\sigma+it)|\bigr) \
    |\hat{\varphi}(-\delta+iu)|du.
\end{align*}
We therefore need to bound
$$
\int_{\Rr}\expect\bigl(|L(-\delta+iu+\sigma+it)|\bigr) \
|\hat{\varphi}(-\delta+iu)|du.
$$
for some fixed $\sigma+it$ in the compact rectangle $R$.
The real part of the argument $-\delta+iu+\sigma+it$ is
$\sigma-\delta\geq \demi+\delta$ by definition of $\delta$, and hence
$$
\expect(|L(-\delta+iu+\sigma+it)|)\ll 1+|-\delta+iu+\sigma+it|\ll
1+|u|
$$
uniformly for $\sigma+it\in R$ and $u\in\Rr$ by
Lemma~\ref{lm-pol-growth}. Since $\hat{\varphi}$ decays faster than
any polynomial at infinity, we conclude that
$$
\int_{\Rr}\expect\bigl(|L(-\delta+iu+\sigma+it)|\bigr) \
|\hat{\varphi}(-\delta+iu)|du \ll 1
$$
uniformly for $s=\sigma+it\in R$, and the result follows.
\end{proof}

We proceed similarly for the $L$-functions.

\begin{lemma}\label{lm-critical}
  For $N\geq 1$ and $f\in S_2(q)^*$, define
$$
L^{(N)}(f,s)=\sum_{n\geq
  1}\lambda_f(n)\varphi\Bigl(\frac{n}{N}\Bigr)n^{-s},
$$
and define $\A{L}_{q}^{(N)}$ to be the $\hol(D)$-valued random
variable mapping $f$ to $L^{(N)}(f,s)$ restricted to $D$.  We then
have
$$
\expect_q(\|\A{L}_{q}-\A{L}_{q}^{(N)}\|_{\infty})\ll N^{-\delta}
$$
for $N\geq 1$ and all $q$.
\end{lemma}

\begin{proof}
  For any $s\in R$, we have the representation
\begin{equation}\label{eq-approx}
  L(f,s)-L^{(N)}(f,s)=-\frac{1}{2i\pi}\int_{(-\delta)}
  L(f,s+w)\hat{\varphi}(w)N^wdw.
\end{equation}
and for any $v$ with $\Reel(v)>1/2$, Cauchy's theorem gives
$$
L(f,v)-L^{(N)}(f,v)=\frac{1}{2i\pi}\int_{\partial R}
(L(f,s)-L^{(N)}(f,s))\frac{ds}{s-v},
$$
where the boundary of $R$ is oriented counterclockwise. As in the
previous argument, we deduce that 
$$
\|\A{L}_{q}-\A{L}_{q}^{(N)} \|_{\infty}\ll \int_{\partial R} |L(f,s)-L^{(N)}
(f,s)||ds|
$$
for $f\in S_2(q)^*$.  Taking the expectation with respect to $f$ and
changing the order of summation and integration leads to 
\begin{align}
  \expect_q\Bigl( \|\A{L}_{q}-\A{L}_{q}^{(N)}\|_{\infty}\Bigr) 
  &\ll
    \int_{\partial R}
    \expect_q\bigl(|L(f,s)-L^{(N)}(f,s)|\bigr)|ds|
    \notag\\
  &\ll  \sup_{s\in \partial R}
    \expect_q\bigl(|L(f,s)-L^{(N)}(f,s)|\bigr).
    \label{eq-last-ref}
\end{align}
Applying~(\ref{eq-approx}) and using again Fubini's Theorem, we obtain
$$
\expect_q\bigl(|L(f,s)-L^{(N)}(f,s)|\bigr) \ll N^{-\delta} \int_{\Rr}
|\hat{\varphi}(-\delta+iu)|
\expect_q\bigl(|L(f,-\delta+iu+\sigma+it)|\bigr)du
$$
for $s\in\partial R$. Since $\sigma-\delta \geq \demi+\delta$, we get
\begin{equation}\label{eq-zeta-mean}
  \expect_q\bigl(|L(f,-\delta+iu+\sigma+it)|\bigr)
  \ll (1+|u|)^A
\end{equation}
by Proposition~\ref{pr-second-moment}, where $A$ is an absolute
constant. Hence
\begin{equation}\label{eq-zeta-mean-app}
  \expect_q\bigl(|L(f,s)-L^{(N)}(f,s)|\bigr) \ll N^{-\delta}
  \int_{\Rr}|\hat{\varphi}(-\delta+iu)|(1+|u|)^Adu\ll N^{-\delta}.
\end{equation}
We conclude from~(\ref{eq-last-ref}) that
$$
\expect_q\Bigl( \|\A{L}_{q}-\A{L}_{q}^{(N)}\|_{\infty}\Bigr) \ll
N^{-\delta},
$$
as claimed.
\end{proof}

\begin{proof}[Proof of Theorem~\ref{th-equidistribution}]
  A simple consequence of the definition of convergence in law shows
  that it is enough to prove that for any bounded and Lipschitz
  function $f\,:\, \hol(D)\lra \Cc$, we have
$$
\expect_q(f(\A{L}_{q}))\lra \expect(f(L_D))
$$
as $q\ra +\infty$ (see~\cite[p. 16, (ii)$\Rightarrow$ (iii) and (1.1),
p. 8]{billingsley}). To prove this, we use the Dirichlet series
expansion of $L_D$ given by Lemma~\ref{lm-random} (2).
\par
Let $N\geq 1$ be some integer to be chosen later. Let
$$
\A{L}_{q}^{(N)}=\sum_{n\geq
  1}\lambda_f(n)n^{-s}\varphi\Bigl(\frac{n}{N}\Bigr)
$$
(viewed as random variable defined on $S_2(q)^*$) and
$$
L_{N}=\sum_{n\geq 1}Y_n n^{-s}\varphi\Bigl(\frac{n}{N}\Bigr)
$$
be the smoothed partial sums of the Dirichlet series, as in
Lemmas~\ref{lm-critical} and~\ref{lm-random-critical}.
\par
We then write
\begin{multline*}
  |\expect_q(f(\A{L}_{q}))-\expect(f(L))| \leq
  |\expect_q(f(\A{L}_{q})-f(\A{L}_{q}^{(N)}))|+\\
  |\expect_q(f(\A{L}_{q}^{(N)}))-\expect(f(L^{(N)}))|+ |\expect(f(L^{(N)})-f(L))|.
\end{multline*}
\par
Since $f$ is a Lipschitz function on $\hol(D)$, there exists a
constant $C\geq 0$ such that
$$
|f(x)-f(y)|\leq C\|x-y\|_{\infty}
$$
for all $x$, $y\in\hol(D)$. Hence we have
\begin{multline*}
  |\expect_q(f(\A{L}_{q}))-\expect(f(L))| \leq
  C\expect_q(\|\A{L}_{q}-\A{L}_{q}^{(N)}\|_{\infty})+\\
  |\expect_q(f(\A{L}_{q}^{(N)}))-\expect(f(L^{(N)}))|+
  C\expect(\|L^{(N)}-L\|_{\infty}).
\end{multline*}
\par
Fix $\eps>0$.  Lemmas~\ref{lm-critical} and~\ref{lm-random-critical}
together show that there exists some $N\geq 1$ such that
$$
\expect_q(\|\A{L}_{q}-\A{L}_{q}^{(N)}\|_{\infty})<\eps
$$
for all $q\geq 2$ and
$$
\expect(\|L^{(N)}-L\|_{\infty})<\eps.
$$
We \emph{fix} such a value of $N$. By Proposition~\ref{pr-spectral}
(and composition with a continuous function), the random variables
$\A{L}_{q}^{(N)}$ (which are Dirichlet polynomials) converge in law to
$L^{(N)}$ as $q\ra +\infty$. We deduce that we have
$$
|\expect_q(f(\A{L}_{q}))-\expect(f(L))| < 4\eps
$$
for all $q$ large enough. This finishes the proof.
\end{proof}

\section{Proof of Theorem~\ref{th-support}}
\label{sec-proof2}

For the computation of the support of the random Dirichlet series
$L(s)$, we apply a trick to exploit the analogous result known for the
case of the Riemann zeta function.
We denote $\SUhat$ the product of copies of the unit circle indexed by
primes, so an element $(x_p)$ of $\SUhat$ is a family of matrices in
$\SU_2(\Cc)$ indexed by $p$.

The assumptions on $D$ in Theorem~\ref{th-support}\footnote{\ These
  assumptions could be easily weakened, as has been done for Voronin's
  Theorem.} imply that there exists $\tau$ be such that $1/2<\tau<1$
and $r>0$ such that
$$
D=\{s\in\Cc\,\mid\, |s-\tau|\leq r\}\subset \{s\in\Cc\,\mid\,
1/2<\Reel(s)<1\}.
$$

\begin{lemma}\label{lm-support-1}
  Let $N$ be an arbitrary positive real number.  The set of all series
$$
\sum_{p>N}\frac{\Tr(x_p)}{p^{s}}, \quad\quad (x_p)\in \SUhat
$$
which converge in $\hol(D)$ is dense in the subspace $\hol_{\Rr}(D)$.
\end{lemma}

In the proof and the next, we allow ourselves the luxury of writing
sometimes $\|\varphi(s)\|_{\infty}$ instead of $\|\varphi\|_{\infty}$.

\begin{proof}
  Bagchi~\cite[Lemma 5.2.10]{bagchi} proves (using results of complex
  analysis due to Bernstein, Poly\'a and others) that the set of
  series
$$
\sum_{p>N}\frac{e^{i\theta_p}}{p^{s}}, \quad\quad \theta_p\in\Rr
$$
that converge in $\hol(D)$ is dense in $\hol(D)$ (precisely, he proves
this for $N=1$, but the same proof applies to any value of $N$). If
$\varphi\in\hol_{\Rr}(D)$ and $\eps>0$, we can therefore find real
numbers $(\theta_p)$ such that
$$
\Bigl\| \frac{\varphi(s)}{2}-\sum_{p>N}\frac{e^{i\theta_p}}{p^{s}}
\Bigr\|_{\infty}<\frac{\eps}{2}.
$$
It follows then that
$$
\Bigl\|
\frac{\overline{\varphi(\bar{s})}}{2}-\sum_{p>N}\frac{e^{-i\theta_p}}{p^{s}},
\Bigr\|_{\infty}<\frac{\eps}{2},
$$
hence (since $\varphi\in\hol_{\Rr}(D)$) that
$$
\Bigl\|
\varphi(s)-\sum_{p>N}\frac{e^{i\theta_p}+e^{-i\theta_p}}{p^{s}}
\Bigr\|_{\infty}<\eps,
$$
which gives the result since
$$
e^{i\theta_p}+e^{-i\theta}=\Tr\begin{pmatrix}e^{i\theta_p}&0\\
0&e^{-i\theta_p}
\end{pmatrix}
$$
is the trace of a matrix in $\SU_2(\Cc)$.
\end{proof}

We will use this to prove:

\begin{proposition}
  The support of the law of
$$
\log L_D(s)=-\sum_{p} \log\det(1-X_pp^{-s})
$$
in $\hol(D)$ is $\hol_{\Rr}(D)$.
\end{proposition}

\begin{proof}
  Since $X_p\in \SU_2(\Cc)$, the function $\log L(s)$ is almost surely
  in the space $\hol_{\Rr}(D)$.  Since the summands are independent, a
  well-known result concerning the support of random series (see,
  e.g.,~\cite[Prop. B.8.7]{proba}) shows that it suffices to prove
  that the set of convergent series
$$
-\sum_{p} \log \det(1-x_pp^{-s}),\quad\quad (x_p)\in\SUhat,
$$
is dense in $\hol_{\Rr}(D)$. Denote $L(s;(x_p))$ this series, when it
converges in $\hol(D)$.
\par
We can write
$$
-\sum_{p} \log \det(1-x_pp^{-s})=
\sum_{p}\frac{\Tr(x_p)}{p^s}+g(s;(x_p))
$$
where $s\mapsto g(s;(x_p))$ is holomorphic in the region
$\Reel(s)> 1/2$. Indeed
$$
g(s;(x_p))=\sum_{p}\log\Bigl(1+\sum_{k\geq 0} \Tr(x_p)^kp^{-(k+2)s}
\Bigr).
$$
\par
Fix $\varphi\in\hol_{\Rr}(D)$ and let $\eps>0$ be fixed. There exists
$N\geq 1$ such that
\begin{equation}\label{eq-correct}
  \Bigl\|s\mapsto \sum_{p>N}\log\Bigl(1+\sum_{k\geq 0}
  \Tr(x_p)^kp^{-(k+2)s} \Bigr)\Bigr\|_{\infty}<\eps
\end{equation}
for \emph{any} $(x_p)\in\SUhat$. Now take $x_p=1\in\SU_2(\Cc)$ for
$p\leq N$ and define
$$
\varphi_1=\varphi+\sum_{p\leq
  N}\frac{\Tr(x_p)}{p^s}=\varphi+2\sum_{p\leq N}\frac{1}{p^s},
$$
which belongs to $\hol_{\Rr}(D)$.  By Lemma~\ref{lm-support-1}, there
exist $x_p$ for $p>N$ in $\SU_2(\Cc)$ such that
$$
\Bigl\| \sum_{p>N}\frac{\Tr(x_p)}{p^s}- \varphi_1(s)
\Bigr\|_{\infty}<\eps.
$$
The left-hand side is the norm of
$$
\log L(s;(x_p))-g(s;(x_p))-\sum_{p\leq
  N}\frac{\Tr(x_p)}{p^s}-\varphi_1(s)= \log
L(s;(x_p))-\varphi(s)-g(x;(x_p)),
$$
and by~(\ref{eq-correct}), we obtain
$$
\|\log L(s;(x_p))-\varphi(s)\|_{\infty}<2\eps.
$$
This implies the lemma.
\end{proof}

Using composition with the exponential function and a lemma of Hurwitz
(see, e.g.,~\cite[3.45]{titchmarsh-fn}) on zeros of limits of
holomorphic functions, we see that the support of the limiting
Dirichlet series $L_D$ in $\hol(D)$ is the union of the zero function
and the set of functions $\varphi\in\hol_{\Rr}(D)$ such that
$\varphi(\sigma)>0$ for $\sigma\in D\cap \Rr$. In particular, this
proves Theorem~\ref{th-support}.

\section{Generalizations}
\label{sec-general}

It is clear from the proof that Bagchi's Theorem should hold in
considerable generality for any family of $L$-functions. Indeed, the
crucial ingredients are the local spectral equidistribution
(Proposition~\ref{pr-spectral}), and the first moment estimate
(Proposition~\ref{pr-second-moment}).
\par
The first result is a qualitative statement that is understood to be
at the core of any definition of ``family'' of $L$-functions (this is
explained in~\cite{families}, but also appears, with a different
terminology, for the families of
Conrey--Farmer--Keating--Rubinstein--Snaith~\cite{cfkrs} and
Sarnak--Shin--Templier~\cite{sst}); it is now know in many
circumstances (indeed, often in quantitative form).
\par
The moment estimate is typically derived from a second-moment bound,
and is also definitely expected to hold for a reasonable family of
$L$-functions, but it has only been proved in much more restricted
circumstances than local spectral equidistribution.  However, it is
very often the case that one can at least prove (using local spectral
equidistribution) a weaker statement: for some $\sigma_1$ such that
$1/2<\sigma_1<1$, the second moment of the $L$-functions satisfies the
analogue of Proposition~\ref{pr-second-moment}; an analogue of
Bagchi's Theorem then follows at least for compact discs in the region
$\sigma_1<\Reel(s)<1$.
\par
As far as universality (i.e., Theorem~\ref{th-support}) is concerned,
one may expect that (using tricks similar to the proof of
Theorem~\ref{th-support}) only two different cases really occur,
depending on whether the coefficients of the $L$-functions are real
(as in our case) or complex (as in the case of vertical translates of
a fixed $L$-function).


\begin{thebibliography}{C}

\bibitem{bagchi} B. Bagchi: \textit{Statistical behaviour and
    universality properties of the Riemann zeta function and other
    allied Dirichlet series}, PhD thesis, Indian Statistical
  Institute, Kolkata, 1981; available at
  \url{library.isical.ac.in/jspui/handle/10263/4256}

\bibitem{billingsley}
P. Billingsley: \textit{Convergence of probability measures}, Wiley
(1968). 


\bibitem{cogdell-michel} J. Cogdell and P. Michel: \textit{On the
    complex moments of symmetric power $L$-functions at $s = 1$},
  Internat. Math. Res. Notices (2004), 1561--1617.

\bibitem{cdf} B. Conrey, W. Duke and D. Farmer: \textit{The
    distribution of the eigenvalues of Hecke operators}, Acta
  Arith. 78 (1997), 405--409.

\bibitem{cfkrs} J.B. Conrey, D. Farmer, J. Keating, M. Rubinstein
  and N. Snaith: \textit{Integral moments of $L$-functions},
  Proc. Lond. Math. Soc. 91 (2005) 33--104.


\bibitem{hardy-riesz} G.H. Hardy and M. Riesz: \textit{The general
    theory of Dirichlet's series}, Cambridge Tracts in Math. 18,
  C.U.P. 1915.

\bibitem{ik}
H. Iwaniec and E. Kowalski: \textit{Analytic number theory},
Colloquium Publ. 53, American Math. Soc. 2004.

\bibitem{families} E. Kowalski: \textit{Families of cusp forms},
  Pub. Math. Besançon, 2013, 5--40.

\bibitem{proba} E. Kowalski: \textit{Arithmetic randonnée: an
    introduction to probabilistic number theory}, ETH Zürich Lecture
  Notes,
  \url{www.math.ethz.ch/~kowalski/probabilistic-number-theory.pdf}.

\bibitem{klsw} E. Kowalski, Y-K. Lau, K. Soundararajan and J. Wu:
  \textit{On modular signs}, Math. Proc. Cambridge Phil. Soc. 149
  (2010), 389--411 \url{doi:10.1017/S030500411000040X}

\bibitem{km} E. Kowalski and Ph. Michel: \textit{The analytic rank of
    $J_0(q)$ and zeros fo automorphic $L$-functions}, Duke
  Math. J. 100 (1999), 503--542.


\bibitem{kst2} E. Kowalski, A. Saha and J. Tsimerman: \textit{Local
    spectral equidistribution for Siegel modular forms and
    applications}, Compositio Math. 148 (2012), 335--384.

\bibitem{lm} A. Laurin\v cikas and K. Matsumoto: \textit{The
    universality of zeta-functions attached to certain cusp forms},
  Acta Artih. 98 (2001), 345--359.

\bibitem{li-queffelec}
D. Li and H. Queffélec: \textit{Introduction à l'étude des espaces de
  Banach; Analyse et probabilités}, Cours Spécialisés 12, S.M.F,
2004. 


\bibitem{sst} P. Sarnak, S-W. Shin and N. Templier: \textit{Families
    of $L$-functions and their symmetry}, Proceedings of the 2014
  Simons Trace Formula Symposium, to appear.


\bibitem{serre}
J-P. Serre: \textit{R\'epartition asymptotique des valeurs propres de
  l'op\'erateur de Hecke $T_p$}, J. American Math. Soc. 10 (1997),
75--102. 



\bibitem{titchmarsh-fn} E.C. Titchmarsh: \textit{The theory of
    functions}, 2nd edition, Oxford Univ. Press, 1939.

\bibitem{voronin} S.M. Voronin: \textit{Theorem on the `universality'
    of the Riemann zeta function}, Izv. Akad. Nauk SSSR, 39 (1975);
  475--486; translation in Math. USSR Izv. 9 (1975), 443--445.

\end{thebibliography}
\end{document}